\documentclass[psamsfonts,12pt, draft]{amsart}
\usepackage{amssymb,amsmath}
\usepackage{amsthm}
\usepackage{mathrsfs}
\usepackage{enumerate,indentfirst}
\usepackage[fleqn,tbtags]{mathtools}
\usepackage{color}
\usepackage{courier}
\usepackage{hyperref}
\usepackage[a4paper,left=2.5cm,right=2.5cm,top=3.2cm,bottom=3.2cm]{geometry}
\usepackage{multicol}
\usepackage{xr}
\usepackage{multicol}
 \newtheorem{theorem}{Theorem}[section]
 \newtheorem{corollary}[theorem]{Corollary}
 \newtheorem{lemma}[theorem]{Lemma}
 \newtheorem{proposition}[theorem]{Proposition}

 \theoremstyle{definition}

 \theoremstyle{remark}
 \newtheorem{remark}[theorem]{Remark}
  \numberwithin{equation}{section}

\renewcommand{\phi}{\varphi}
\renewcommand{\theta}{\vartheta}

\DeclareMathOperator{\tform}{\mathfrak{t}}

\DeclareMathOperator{\wform}{\mathfrak{w}}

\DeclareMathOperator{\trace}{Tr}

\DeclarePairedDelimiterX\sipt[2]{(}{)_{\tform}}{#1\,\delimsize\vert\,#2}
\DeclarePairedDelimiterX\sipv[2]{(}{)_{v}}{#1\,\delimsize\vert\,#2}
\DeclarePairedDelimiterX\sipw[2]{(}{)_{w}}{#1\,\delimsize\vert\,#2}

\newcommand{\alg}{\mathscr{A}}

\newcommand{\abs}[1]{\lvert#1\rvert}
\newcommand{\dupN}{\mathbb{N}}

\newcommand{\seq}[1]{(#1_{n})_{n\in\dupN}}

\newcommand{\nen}{n\in\mathbb{N}}

\newcommand{\ran}{\operatorname{ran}}

\newcommand{\finh}{\mathscr{B}_F(\hil)}

\newcommand{\hil}{H}

\newcommand{\bh}{\mathrm{B}(\hil)}

\DeclarePairedDelimiterX\sip[2]{(}{)}{#1\,\delimsize\vert\,#2}
\DeclarePairedDelimiterX\siptilde[2]{(}{)_{\!_{\widetilde{A}}}}{#1\,\delimsize\vert\,#2}
\DeclarePairedDelimiterX\sipf[2]{(}{)_{f}}{#1\,\delimsize\vert\,#2}
\DeclarePairedDelimiterX\sipg[2]{(}{)_{g}}{#1\,\delimsize\vert\,#2}
\DeclarePairedDelimiterX\siptw[2]{(}{)_{\tform+\wform}}{#1\,\delimsize\vert\,#2}
\DeclarePairedDelimiterX\set[2]{\{}{\}}{#1\,\delimsize\vert\,#2}
\DeclarePairedDelimiterX\dual[2]{\langle}{\rangle}{#1,#2}
\DeclarePairedDelimiterX\sipa[2]{(}{)_{\!_A}}{#1\,\delimsize\vert\,#2}
\DeclarePairedDelimiterX\sipc[2]{(}{)_{\!_C}}{#1\,\delimsize\vert\,#2}
\DeclarePairedDelimiterX\sipab[2]{(}{)_{\!_{A+B}}}{#1\,\delimsize\vert\,#2}
\DeclarePairedDelimiterX\sipb[2]{(}{)_{\!_B}}{#1\,\delimsize\vert\,#2}

\newcommand{\hsh}{\mathrm{B}_2(\hil)}
\newcommand{\trh}{\mathrm{B}_1(\hil)}

\allowdisplaybreaks
\begin{document}
\title{On the uniqueness of the Lebesgue decomposition of normal states}
\author[Zs. Tarcsay]{Zsigmond Tarcsay}

\address{%
Department of Applied Analysis\\ E\"otv\"os L. University\\ P\'azm\'any P\'eter s\'et\'any 1/c.\\ Budapest H-1117\\ Hungary}

\email{tarcsay@cs.elte.hu}

\author[T. Titkos]{Tam\'as Titkos}

\address{%
Alfr\'ed R\'enyi Institute of Mathematics\\
Hungarian Academy of Sciences\\Re\'altanoda u. 13-15.\\
Budapest H-1053\\ Hungary}

\email{titkos.tamas@renyi.mta.hu}

%----------classification, keywords, date
%\subjclass{Primary 46L30, 46L51}


%\footnote{ %acknowledgment of support etc. if any
%$^{*}$Thanks.
%}
%%%%%%%%%%%%%%%%%%%%%%%%%%%%%%%%%%%%%%%%%
\maketitle

\begin{abstract}
\textbf{This manuscript draft has never been published and is not submitted for publication. For a substantially revised and extended version see arXiv:1608.03733 [math.FA]. See also arXiv:1710.06830 [math.FA].}\\

The non-commutative theory of the Lebesgue-type decomposition of positive functionals is originated with S.~P.~Gudder. Although H.~Kosaki's counterexample shows that the decomposition is not unique in general, the complete characterization of uniqueness is still not known.

Using the famous operator-decomposition of T.~Ando, we give a necessary and sufficient condition for uniqueness in the particular case when the underlying algebra is $\bh$, the $C^*$-algebra of all continuous linear operators on a Hilbert space $\hil$. Namely, given a normal state $f$, the $f$-Lebesgue decomposition of any other normal state is unique if and only if the representing trace class operator of $f$ has finite rank. 

Some recent results tell that the decomposition is unique over a large class of commutative algebras. Our characterization demonstrates that the lack of commutativity is not the real cause of non-uniqueness.
\end{abstract}

\section*{Introduction} %delete * to number this section
An important noncommutative generalization of the classical Lebesgue decomposition is originated with S. P. Gudder \cite{Gudder}. He proved that any positive functional on a complex unital Banach $^*$-algebra admits a Lebesgue decomposition with respect to any other positive functional (on the same algebra). His result has been recently extended to representable functionals on any  $^*$-algebra \cite{Tarcsay_repr}, and to representable forms on a complex algebra \cite{Szucs2013}. However, Gudder mentioned that he has not been able to prove the unicity of the corresponding decomposition. Few years later H. Kosaki \cite{Kosaki} provided a counterexample demonstrating that such a decomposition does not need to be unique, not even considering normal states on a von Neumann algebra. 

Nevertheless, Kosaki's example does not give any criterion for deciding whether the Lebesgue decomposition of a given normal state relative to another is unique or not. The main goal of this note is to provide a complete solution to this problem among normal states of the von Neumann algebra $\bh$. The two cornerstones of our approach are: Theorem \ref{T:maintheorem} below describing the very close connection between the corresponding  absolute continuity and singularity concepts for normal functionals and their representing operators; and T. Ando's result \cite[Theorem 6]{Ando} characterizing the uniqueness of the Lebesgue decomposition among bounded positive operators.

Let us start by recalling some basic definitions of non-commutative Lebesgue decomposition theory, developed by Gudder \cite{Gudder} (for the more general setting of $^*$-algebras the reader is referred to \cite{Tarcsay_repr}). Suppose we are given a unital Banach $^*$-algebra $\alg$ and a positive functional $f$ on it. Positivity here is understood to mean that $f(a^*a)\geq0$ for all $a\in\alg$. Given another positive functional $g$ on $\alg$, following the terminology of Gudder \cite{Gudder}, we say that $g$ is $f$-\emph{dominated} if there exists $M\geq 0$ such that $g(a^*a)\leq  Mf(a^*a)$ for all $a\in\alg$. Furthermore, $g$ is called $f$-\emph{closable} (or  strongly $f$-\emph{absolutely continuous}, see \cite{Gudder}) if for any sequence $\seq{a}$ of $\alg$
\begin{equation*}
f(a_n^*a_n)\to0\qquad \textrm{and}\qquad g((a_n-a_m)^*(a_n-a_m))\to 0
\end{equation*} 
imply  $g(a_n^*a_n)\to0$. Finally, $f$ and $g$ are called \emph{mutually singular} if for any positive functional $h$ on $\alg$, $h\leq f$ and $h\leq g$ imply $h=0$. We mention here that Gudder \cite{Gudder} used a non-symmetric singularity concept, called semi-singularity. Because of equivalency (see \cite[Theorem 5.2]{Tarcsay_parallel}), we use the symmetric one instead. A Lebesgue type decomposition of $g$ relative to $f$ is a pair $(g_0,g_1)$ of positive functionals such that $g=g_0+g_1$, where $g_0$ is $f$-closable and $g_1$ is $f$-singular. Such a decomposition always exists, according to \cite[Corollary 3]{Gudder}. Moreover, in view of \cite[Theorem 3.3]{Tarcsay_repr}, there is a Lebesgue decomposition $(g_r, g_s)$ which is extremal in the sense that $h\leq g_r$ holds for any $f$-closable positive functional $h$ with $h\leq g$. Here, $g_r$ is called the $f$-\emph{regular part} of $g$.

The word \lq\lq closable" refers to the fact that a functional $g$ is $f$-closable precisely when considering the GNS-triplets $(\hil_f,\pi_f,\zeta_f)$ and $(\hil_g,\pi_g,\zeta_g)$, the correspondence $\pi_f(a)\zeta_f\mapsto\pi_g(a)\zeta_g$ defines a closable linear operator between $\hil_f$ and $\hil_g$ (cf. \cite{Tarcsay_repr}). One major advantage of this observation is that one can apply unbounded operator techniques to treat the Lebesgue decomposition and the Radon--Nikodym type theory, see eg. \cite{Araki,Sakai,Tarcsay_RN, Tarcsay_repr}. In the present paper we use an equivalent for the property of $f$-closability, namely the notion of "almost domination" as follows: the positive functional $g$ is \emph{almost dominated} by $f$ (in notion, $g\ll f$) if $g$ is the pointwise limit of a sequence $\seq{g}$ of positive functionals possessing the properties $g_n\leq g_{n+1}\leq g$, and $g_n\leq c_nf$ for some $c_n\geq0$. The equivalence of these concepts can be proved using to the notion of \emph{parallel sum} of two positive functionals (\cite[Theorem 5.1]{Tarcsay_parallel}  and due to \cite[Theorem 3.8]{Hassi2009}; see also \cite[Theorem 2.15]{Szucs_abs}).

\section{Closability and singularity of positive functionals on $\bh$}

Given a complex Hilbert space $\hil$, we denote by $\bh$  the $C^*$-algebra of continuous linear operators on $\hil$, and by $\hsh, \trh$ and $\finh$  the ideals of  Hilbert--Schmidt, trace class, and continuous finite rank operators in $\bh$. Recall that $\hsh$ is a complete Hilbert algebra with respect to the scalar product
\begin{eqnarray}\label{E:siphsh}
\sip{S}{T}_2=\trace(T^*S)=\sum_{e\in E}\sip{Se}{Te},\qquad S,T\in\hsh,
\end{eqnarray}
where $\trace$ refers to the the trace functional and $E$ is any orthonormal basis in $\hil$. Furthermore, $\trh$ is a Banach $^*$-algebra  under the norm $\|A\|_1:=\trace (\abs{A})$. It is also well known that $\finh$ is dense in both $\hsh$ and $\trh$, with respect to the norms $\|\cdot\|_2$ and $\|\cdot\|_1$, respectively. Recall also that $A\in\trh$ holds if and only if $A$ is the product of two Hilbert--Schmidt operators. For further basic properties of Hilbert-Schmidt and trace class operators we refer the reader to \cite{KR1, Murphy, palmer, Pedersen}.

For any  $T\in\trh$  we set 
\begin{eqnarray}\label{E:f_T}
f_T(A):=\trace(AT)=\trace(TA),\qquad A\in\bh,
\end{eqnarray}
which defines a continuous linear functional on $\bh$ due to inequality 
\begin{eqnarray}\label{trace inequality}
\abs{\trace(AT)}\leq \|A\|\|T\|_1.
\end{eqnarray}
Functionals of this type are called \emph{normal functionals}. 

It is not difficult to prove that  $f_T$ is positive if and only if $T$ is positive
Before establishing the statement we present two lemmas which are consequences of \cite[Theorem 5.6]{SSZT} concerning positive extendibility of linear functionals. 
 
\begin{lemma}\label{L:fS=fN}
Let $f$ be a normal positive functional on $\bh$ and denote by $\phi$ the restriction of $f$ to the ideal $\finh$ of finite rank operators. Then $f$ is the smallest among all positive functionals extending $\phi$. (In the language of \cite{SSZT}, $f$ equals the Krein-von Neumann extension $\phi_N$ of $\phi$.)
\end{lemma}
%\begin{proof}
%By  \cite[Theorem 5.6]{SSZT} there exists a positive functional $\phi_N$ on $\bh$ extending $\phi$ such that  $\phi_N\leq\widetilde{\phi}$ for any positive extension $\widetilde{\phi}$ of $\phi$ to $\bh$. It suffices therefore to show that $f_T=\phi_N$. As $f_T-\phi_N\geq0$ due to the minimality of $\phi_N$, that will follow by showing that $\phi_N(I)\geq f_T(I)=\trace(T)$, where $I$ refers to the identity operator in $\hil$. From the proof of  \cite[Theorem 5.6]{SSZT} it also turns out that 
%\begin{equation*}
%\phi_N(X^*X)=\sup\set{\abs{\phi(X^*A)}^2}{A\in\finh, \phi(A^*A)\leq1}
%\end{equation*}
%for any $X\in\bh$. Choosing $A=\trace(T)^{-1/2}P$ for any   projection $P$ with finite rank, we see that $\phi(A^*A)=\trace(T)^{-1}\trace(PT)\leq1$, whence
%\begin{align*}
%\phi_N(I)\geq \abs{\phi(A)}^2=\frac{\trace(PT)^2}{\trace(T)}.
%\end{align*}
%Taking supremum in $P$ on the right side gives $\phi_N(I)\geq\trace(T)$.
%\end{proof}
%
\begin{lemma}\label{L:g=fS}
Let $f$ be any positive functional on $\bh$ and let $\phi$ denote its restriction to $\finh$.
\begin{enumerate}[\upshape a)]
 \item The smallest positive (Krein-von Neumann) extension of $\phi$ to $\bh$ is always normal.
 \item If there exists a normal positive functional $g$ with $f\leq g$ then $f$ is normal too. 
\end{enumerate}
\end{lemma}
Recall now the notions of absolute continuity and singularity among positive operators due to Ando \cite{Ando}. A positive operator $S\in\bh$ is called \emph{absolutely continuous} with respect to another positive operator $T\in\bh$ (in notion $S\ll T$) if there is a monotone increasing sequence $\seq{S}$ of positive operators such that  $S_n\leq c_n T$ for some $c_n\geq0$ (i.e., $S_n$ is $T$-dominated for all $n\in\mathbb{N}$) and $S_n\to S$ in the strong operator topology. We say that $S$ and $T$ are \emph{singular} with respect to each other (in notion $S\perp T$) if $R\leq S$ and $R\leq T$ imply $R=0$ for any positive operator $R$. Ando's decomposition theorem (\cite[Theorem 2]{Ando}) asserts that any $S$ splits into a $T$-absolutely continuous part $[T]S$ and a $T$-singular part $S-[T]S$ where $[T]S$ possesses the extremal property that $R\ll T$ and $R\leq S$ imply $R\leq [T]S$. 

A natural question arises here: what is the connection between the corresponding regularity and singularity concepts of trace class operators, and their induced normal states? The answer (which is given in the following theorem) plays a key role by solving the uniqueness problem.

\begin{theorem}\label{T:maintheorem}
Let $S,T$ be positive trace class operators on $\hil$. 
\begin{enumerate}[\upshape a)]
 \item $S$ and $T$ are mutually singular precisely if $f_S$ and $f_T$ are.
 \item $S$ is absolutely continuous with respect to $T$ precisely if $f_S$ is almost dominated by $f_T$.
\end{enumerate}
\end{theorem}
\begin{proof}
 Let us prove a) first: assume that $S$ and $T$ are mutually singular, and let $g$ be any positive functional on $\bh$ such that $g\leq f_T, f_S$. Then, by Lemma \ref{L:g=fS}, $g=f_R$ for some positive trace class operator $R$, and clearly, $R\leq S,T$. Hence $R=0$ by singularity of $S$ and $T$, which yields $g=0$. Conversely, if $f_T$ and $f_S$ are mutually singular and $R\in\bh$ is any positive operator with $R\leq S$ and $R\leq T$, then $R$ must be a trace class operator satisfying $f_R\leq f_S,f_T$. Consequently, $f_R=0$ and thus $R=0$. 
 
 Let us turn to the proof of b): assume first that $f_S$ is almost dominated by $f_T$, and consider a monotone increasing sequence $\seq{g}$ of $f_T$-dominated representable positive functionals such that $\lim\limits_{n\to\infty} g_n(A)=f_S(A)$ for all $A\in\hsh$. Then, due to Lemma \ref{L:g=fS}, $g_n=f_{S_n}$ for any $n$ with some positive operator $S_n\in\trh$. Clearly, $S_n\leq S_{n+1}\leq S$, and $S_n\leq c_nT_n$. Thus the  $T$-absolute continuity of $S$ follows once we prove that 
\begin{eqnarray}\label{E:SntoS}
\sip{S_nx}{x}\to \sip{Sx}{x} \qquad \textrm{for all $x\in\hil$}.
\end{eqnarray}
To see this, let $e\in\hil,\|e\|=1$, and denote by $P$ the orthogonal projection onto the one-dimensional subspace spanned by $e$. Then 
\begin{align*}
\sip{S_ne}{e}=\trace(S_nP)=f_{S_n}(P)\to f_S(P)=\sip{Se}{e},
\end{align*}
which gives \eqref{E:SntoS}. Assume conversely that $S$ is $T$-absolutely continuous. Consider a monotone increasing sequence $\seq{S}$ of $T$-dominated positive operators such that $\sip{S_nx}{x}\to \sip{Sx}{x}$ for each $x\in\hil$. It is clear by $S_n\leq S$ that $S_n\in\trh$ and that $f_{S_n}\leq f_{S_{n+1}}\leq f_{S}$, $f_{S_{n}}\leq c_nf_T$. To conclude that $f_S$ is $f_T$-almost dominated it is enough to prove that 
\begin{align}\label{E:fSntofS}
f_{S_n}(A)\to f_S(A)\qquad \textrm{for all $A\in\bh$}.
\end{align}
With this aim, let $\varepsilon>0$ and choose an orthonormal basis $E$ in $\hil$. Fix a finite subset $F(=F(\varepsilon))$ of $E$ and an integer $N(=N(\varepsilon,F))$ such that 
\begin{align*}
\trace(S)-\sum_{e\in F}\sip{Se}{e}<\frac{\varepsilon}{2}\quad\mbox{and}\quad
\sum_{e\in F}\sip{(S-S_N)e}{e}<\frac{\varepsilon}{2}.
\end{align*}
Then for any integer $n$ with $n\geq N$ we infer that 
\begin{align*}
\trace(S-S_n)\leq \trace(S-S_N)\leq \sum_{e\in E\setminus F} \sip{Se}{e}+\sum_{e\in F}\sip{(S-S_N)e}{e}<\frac{\varepsilon}{2}+\frac{\varepsilon}{2}=\varepsilon.
\end{align*}
Hence, by inequality (\ref{trace inequality}), $f_{S_n}(A)\to f_{S}(A)$ for all $A\in\mathscr{B}(H)$.
\end{proof}

\section{On the uniqueness of Lebesgue decomposition of normal states}

 From Theorem \ref{T:maintheorem} above we conclude immediately that the transformation $T\mapsto f_T$ between positive trace class operators and normal positive functionals preserves the Lebesgue decomposition. That is to say, $S=S_1+S_2$ is a Lebesgue decomposition relative to $T$ if and only if $f_S=f_{S_1}+f_{S_2}$ is a Lebesgue decomposition relative to $f_T$. Moreover, considering two normal positive functionals  $f,g$ on $\bh$, the $f$-regular part $g_r$ of $g$ is the image of $[T]S$ under that map, where $f=f_T$ and $g=f_S$:
\begin{equation*}
g_r=f_{[T]S}.
\end{equation*}
Hence, as a straightforward consequence of  \cite[Theorem 6]{Ando} %(cf. also \cite[Theorem 5]{Tarcsay_Leb}) 
we can establish the following uniqueness result.
\begin{corollary}
 Let $f,g$ be normal positive functionals on $\bh$ with representing operators $T$ and $S$, respectively. The Lebesgue decomposition of $g$ relative to $f$ is unique precisely if $g_r\leq cf$, or equivalently, if $[T]S\leq cT$ for some $c>0$. 
\end{corollary}

However, it is not clear from the preceding corollary whether the Le\-besgue decomposition among  normal positive functionals is unique. In  Theorem \ref{T:nonunique} below  we are going to provide a complete answer to this problem. But first we make a short observation to be used mainly in the proof of the main result: 
\begin{proposition}
Given  a sequence $\seq{\lambda}\in\ell^1$  of positive numbers, there is another sequence $\seq{\mu}\in\ell^1$ of positive numbers such that $(\mu_n/\lambda_n)_{\nen}\notin \ell^{\infty}$.
\end{proposition}
\begin{proof}
Assume indirectly that there is no $(\mu_n)_{n\in\dupN}$ with the prescribed properties. It follows then that for any $\seq{x}$ in $\ell^1$ we have $(x_n/\lambda_n)_{\nen}\in \ell^{\infty}$, or in other words the one-to-one continuous mapping $A:\ell^{\infty}\to \ell^1, \seq{x}\mapsto (\lambda_n x_n)_{\nen}$ is onto, and hence a linear homeomorphism thanks to the Banach open mapping theorem. But this is impossible (as $\ell^1$ is separable and $\ell^{\infty}$ is not).
\end{proof}

\begin{theorem}\label{T:nonunique}
Let $f$ be a normal positive functional on $\bh$ and let $T\in\trh$ stand for its representing operator. 
\begin{enumerate}[\upshape a)]
 \item If $T$ has finite rank then any normal positive functional $g$  possesses a unique $f$-Lebesgue decomposition.
 \item If $T$ has infinite rank then there exists  a normal positive functional $g$  such that the $f$-Lebesgue decomposition of $g$ is not unique. 
\end{enumerate}
\end{theorem}
\begin{proof}
a) Suppose $f$ is a normal state on $\bh$ with finite rank representing operator $T\in\trh$. Then there exist a finite orthonormal system $e_1,\ldots, e_n$ and $\lambda_1,\ldots,\lambda_n$ positive numbers such that  
\begin{align}
Tx=\sum_{k=1}^{n}\lambda_k\sip{x}{e_k}e_k,\qquad \textrm{for all $x\in\hil$.}
\end{align}
Let $g$ be any normal state with representing operator $S\in\trh$ and consider the $T$-Lebesgue decomposition $S=[T]S+(S-[T]S)$  according to Ando \cite{Ando}. Here, the $T$-closable part $[T]S$ satisfies 
\begin{align*}
\ran ([T]S)^{1/2}\subseteq\ker([T]S)^{\perp}\subseteq\ker T^{\perp}=\ran T^{1/2},
\end{align*}  
whence we infer that $[T]S\leq cT$ with appropriate $c\geq0.$ By Ando's uniqueness theorem (\cite[Theorem 3]{Ando}), the $T$-Lebesgue decomposition of $S$ is unique. In the view of Lemma \ref{L:g=fS} and Theorem \ref{T:maintheorem}, the $f$-Lebesgue decomposition of $g$ must be unique as well. \\
b) Consider now a normal state $f$ with infinite rank representing operator $T\in\trh$. By the Hilbert--Schmidt theory of compact selfadjoint operators there exists an orthonormal sequence $\seq{e}$ in $\hil$ and a sequence $\seq{\lambda}$ in $\ell^1$ such that 
\begin{align*}
Tx=\sum_{n=1}^{\infty}\lambda_n\sip{x}{e_n}e_n,\qquad \textrm{for all $x\in\hil$.}
\end{align*}
Choose a sequence $\seq{\mu}\in\ell^1$ of positive numbers such that  $(\mu_n/\lambda_n)_{\nen}\notin\ell^{\infty}$ and consider $S\in\trh$ defined by    
\begin{align*}
Sx=\sum_{n=1}^{\infty}\mu_n\sip{x}{e_n}e_n,\qquad \textrm{ $x\in\hil$.}
\end{align*}
An easy calculation shows that the operator sequence $\seq{S}$,
\begin{align*}
S_nx=\sum_{k=1}^{n}\mu_k\sip{x}{e_k}e_k,\qquad x\in\hil,
\end{align*}
fulfills 
\begin{equation}\label{E:S_n}
0\leq S_n\leq S_{n+1},\qquad \|S_n-S\|\to0,\qquad S_n\leq c_nT,
\end{equation}
for each integer $n$ and for suitable $c_n\geq0$. Hence, $S=[T]S$. Observe that $\mu_n\leq c_n\lambda_n$ holds for each $n$, hence $\seq{c}$ is necessarily unbounded. Hence, although $S$ is $T$-closable due to \eqref{E:S_n}, it cannot be $T$-dominated. Consequently,  the $T$-Lebesgue decomposition of $S$ is not unique by Ando's result \cite[Theorem 6]{Ando}. Thus the $f$-Lebesgue decomposition of $g:=f_S$ also fails to be unique due to Theorem \ref{T:maintheorem}. 
\end{proof}
\begin{remark}
The commutative Gelfand--Naimark theorem and the Riesz representation theorem jointly show that the Lebesgue decomposition on a commutative $C^*$-algebra is always unique, see \cite{Szucs_abs}. One might therefore expect  that non-commutativity is responsible for the absence of uniqueness. But this is not the case: Theorem \ref{T:nonunique} a) says in particular that the Lebesgue decomposition on $\bh$ is necessarily unique if $\hil$ is finite dimensional. So it would be nice to know which property of the underlying algebra actually results uniqueness and which one causes non-uniqueness for the decomposition. 
\end{remark}
%%%%%%%%%%%% References %%%%%%%%%%%%%
%%
%<Author name> is written as Initial of Given Name, and Family Name.
%<Title> is written in roman letters.
%<Journal name> should be abbreviated according to
% the MR Serials Abbreviations List of Mathematical Reviews:
% (Abbreviations of Names of Serials; http://www.ams.org/mr-database)
%For <Pages>, use en-dash "--" between page numbers.
%%

\bigskip
%%%%%%%%%%%% Authors' addresses %%%%%%%%%%%%%
\address{ % First Author
Zolt\'an Sebesty\'en \\
Department of Applied Analysis\\ E\"otv\"os L. University\\ P\'azm\'any P\'eter s\'et\'any 1/c.\\ Budapest H-1117\\ Hungary
}
{sebesty@cs.elte.hu}
\address{% Second Author
Zsigmond Tarcsay \\
Department of Applied Analysis\\ E\"otv\"os L. University\\ P\'azm\'any P\'eter s\'et\'any 1/c.\\ Budapest H-1117\\ Hungary
}
{tarcsay@cs.elte.hu}

%start a new line
\address{% Third Author
Tam\'as Titkos\\
Alfr\'ed R\'enyi Institute of Mathematics\\
Hungarian Academy of Sciences\\Re\'altanoda u. 13-15.\\
Budapest H-1053\\ Hungary
}
{titkos.tamas@renyi.mta.hu}

\end{document}